\documentclass[leqno]{aadmbook}
\usepackage{amsthm}
\usepackage{amsfonts}
\usepackage{amsmath}
\usepackage{latexsym,amsfonts,mathptm}
\usepackage{epsfig,float,graphicx}
\restylefloat{figure} 
\def\Dj{\rlap{-}D}

\newtheorem{thm}{Theorem}
\newtheorem{lem}{Lemma}

\newtheorem{df}{Definition}

\newtheorem{exm}{Example}

\def\ds{\displaystyle}
\def\dzn{,\kern-0.1em,}

\input cyracc.def
 
\def\be{\begin{equation} }
\def\ee{\end{equation} }
\def\bfl{\begin{flushleft} }
\def\efl{\end{flushleft} }
\def\bfr{\begin{flushright} }
\def\efr{\end{flushright} }
\def\bc{\begin{center}}
\def\vs*{\vspace*}
\def\hs*{\hspace*}
\def\ec{\end{center}}
\def\beq{\begin{eqnarray}}
\def\eeq{\end{eqnarray}}

\def\ben{\begin{enumerate}}
\def\een{\end{enumerate}}
\def\bit{\begin{itemize}}
\def\eit{\end{itemize}}

\begin{document}


\oddsidemargin 16.5mm
\evensidemargin 16.5mm

\thispagestyle{plain}



\vspace{5cc}
\begin{center}
{\large\bf  The Structure of the  2-factor Transfer Digraph  common for \\ Thin Cylinder, Torus and Klein Bottle Grid Graphs
\rule{0mm}{6mm}\renewcommand{\thefootnote}{}
\footnotetext{\scriptsize 2010 Mathematics Subject Classification.
05C38, 05C50, 05A15, 05C30, 05C85.

\rule{2.4mm}{0mm}Keywords and Phrases:  2-factor, Hamiltonian cycles, transfer matrix, grid graphs}}

\vspace{1cc} {\large\it   Jelena \Dj oki\' c, Ksenija Doroslova\v
cki   and  Olga Bodro\v{z}a-Panti\'{c}}

\vspace{1cc}
\parbox{24cc}{{\small
We prove that
 the  transfer digraph ${\cal D}^*_{C,m}$
needed for the  enumeration of 2-factors  in  the
 thin cylinder $TnC_{m}(n)$, torus $TG_{m}(n)$ and Klein bottle $KB_m(n)$
(all grid graphs of the fixed width $m$ and with $m \cdot n$ vertices),  when $m$ is odd,
  has   only two components of order $2^{m-1}$ which  are isomorphic.
When $m$  is even, ${\cal D}^*_{C,m}$ has $\ds \left\lfloor    \ds \frac{m}{2} \right\rfloor + 1$ components
 which orders can be expressed via binomial coefficients and
all but one of the components  are bipartite digraphs.
The proof is based on the application of  recently obtained
  results concerning  the related transfer digraph for  linear  grid graphs (rectangular, thick  cylinder and Moebius strip).
  }}
\end{center}


\vspace{1.5cc}
\begin{center}
{\bf  1.  INTRODUCTION}
\end{center}
\label{sec:intro}

\vspace*{4mm}

Although the research related to the enumeration of  Hamiltonian cycles on special classes of  grid graphs of fixed width, such as
Cartesian products of paths and/or cycles, was initiated  more than thirty years ago \cite{TBKS}, there still remain
 many  open  questions  whose answers  should be  sought  in the structure of so-called \emph{transfer digraphs} - auxiliary digraphs using which the counting of the required objects is performed. For more details see \cite{BKP1,BKDP1,BKDjDP}.

 2-factors are the  natural generalization of the concept of Hamiltonian cycles. For a graph $G$, \emph{ 2-factor } is defined  as a spanning subgraph of $G$
 where each vertex has exactly two neighbors. Obviously, it represents the spanning  union of cycles. In the  special case when we have just one cycle, this 2-factor is called Hamiltonian cycle. The systematic study on 2-factors of the mentioned   classes of grid graphs has recently   begun \cite{DjBD1,DjDB2,DjDB3}
in order to help in solving  the above  questions.
The first results were obtained for so-called \emph{linear grid  graphs} of width $m \in N$ (rectangular grid graphs, thick  cylinders and Moebius strips) - the grid graphs whose any subgraph  induced by the vertices from the same columns is the path $P_m$ \cite{DjBD1,DjDB2}.
The subject of interest in this paper are the grid graphs whose any subgraph  induced by the vertices from the same columns is the  cycle  $C_m$.
These are the grid graphs from the title and we refer to them as \emph{circular grid  graphs}.
Since in the process of forming a torus  grid or Klein bottle grid graph, before gluing the ends of the initial tube (thin cylinder grid), one of these ends   can be twisted,
  we obtain more types of such grid graphs.

\begin{figure}[H]
\begin{center}
\includegraphics[width=5.4in]{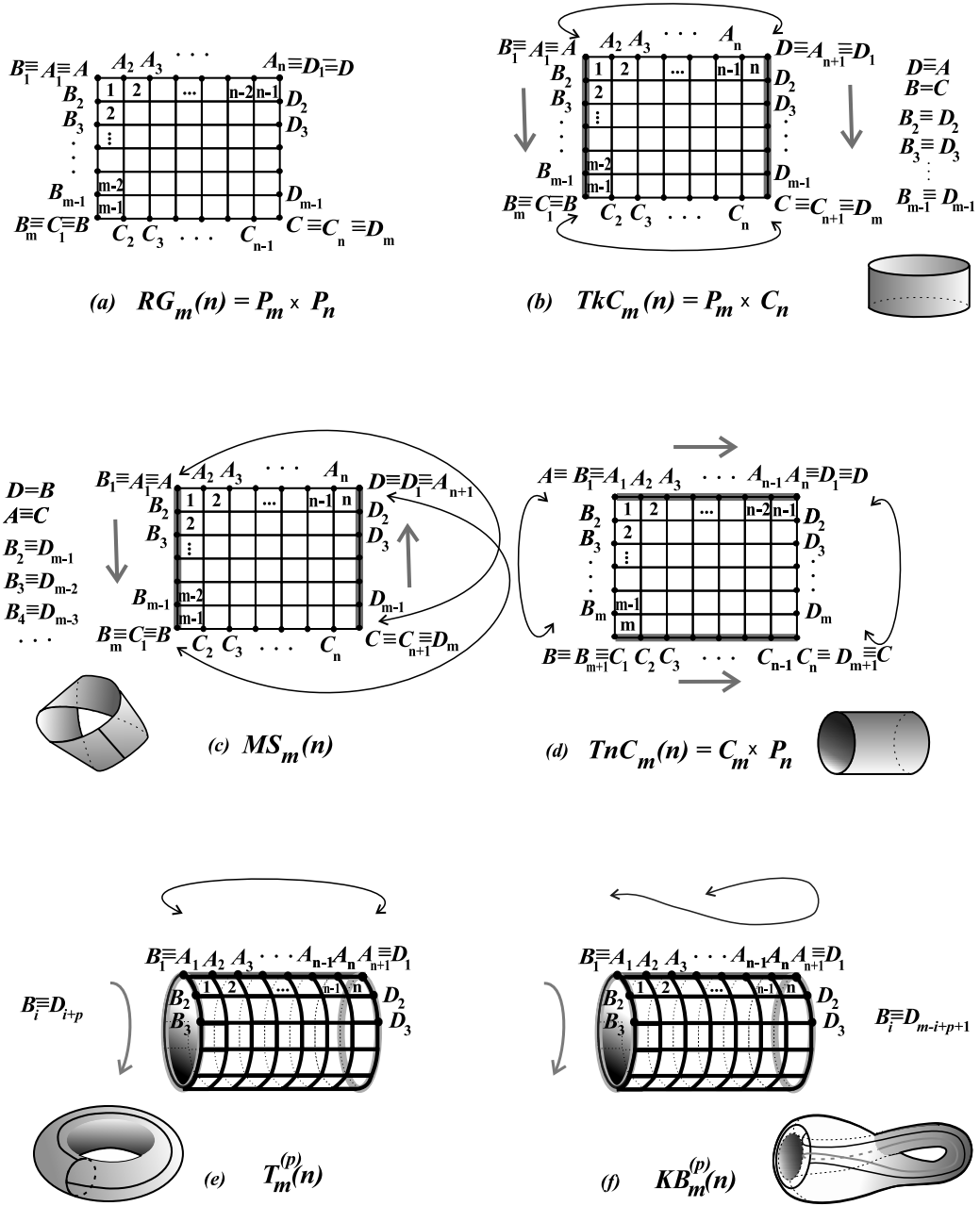}
\\ \ \vspace*{-18pt}
\end{center}
\caption{ {\bf (a)} The rectangular grid  $RG_m(n) = P_m \times P_n$;
{\bf (b)} The thick cylinder  $TkC_m(n) = P_m \times C_n$;
{\bf (c)} The Moebius strip  $MS_m(n)$;
{\bf (d)} The thin cylinder  $TnC_m(n) = P_m \times C_n$;
{\bf (e)} The torus grid   $TG^{(p)}_m(n)$;
{\bf (f)} The Klein bottle  $KB^{(p)}_m(n)$}
\label{SCiMS}
\end{figure}
\unskip

\begin{df} \label{def:grafovi} $ $

The \textbf{rectangular (grid) graph $RG_m(n)$}, \textbf{ thin (grid) cylinder}  $TnC_m(n)$ and \textbf{  thick (grid) cylinder $TkC_m(n)$} ($m,n \in N$) are
 $P_m \times P_{n}$, $ C_m \times P_n$ and  $P_m \times C_n$, respectively.

The \textbf{Moebius strip $MS_m(n)$}  is obtained from $RG_m(n+1) = P_m \times P_{n+1}$
  by identification  of  corresponding vertices from the first and  last  column in the opposite direction without duplicating edges.

  The \textbf{torus (grid)   $TG^{(p)}_{m}(n)$} ($0 \leq p \leq m-1$) is the graph obtained from $TnC_m(n+1)$ by identification of the vertices $B_i $ and $D_{i+p}$,
 $i=1, \ldots , m$, without duplicating edges, where  $B_i$  and $D_i$ denote  the  vertices belonging to the $i$-th row ($1\leq i \leq m$) from  the first and last
column,  respectively  (the sign $+$ in subscript is addition  modulo $m$).

The  \textbf{Klein bottle $KB^{(p)}_m(n)$} is the graph obtained from $TnC_m(n+1)$ by identification of the vertices $B_i $ and  $ D_{m-i+p+1}$, $i=1, \ldots , m$,   without duplicating edges.

The value $m \in N$ is called the \textbf{width} of the grid graph.
The grid graphs $RG_m(n)$, $TkC_m(n)$ and $MS_m(n)$ are uniformly  called the \textbf{ linear grid graphs}, whereas $TnC_m(n)$,  $TG^{(p)}_{m}(n)$ and  $KB^{(p)}_m(n)$
 are   called the \textbf{circular grid graphs}. \end{df}

The 2-factor of the Klein bottle $KB^{(1)}_4(3)$ depicted in  Figure~\ref{primer} (a)
consists of $2$ cycles, while the one of the  torus grid $TG^{(0)}_4(3)$ in  Figure~\ref{primer} (b)
has only one cycle and hence it is Hamiltonian cycle.

\begin{figure}[H]
\begin{center}
\includegraphics[width=3.5in]{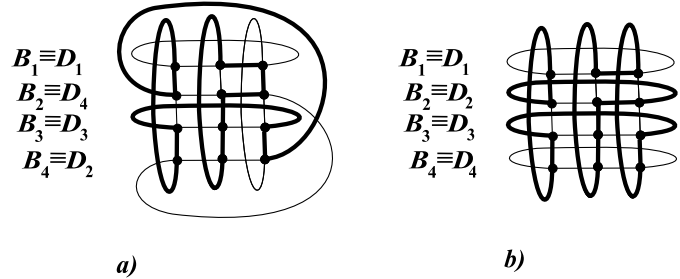}
\\ \ \vspace*{-18pt}
\end{center}
\caption{ {\bf (a)} Klein bottle $KB^{(1)}_4(3)$ with a 2-factor;  {\bf  (b)}  Torus grid $TG^{(0)}_4(3)$ with a Hamiltonian cycle}.
\label{primer}
\end{figure}
\unskip

Observe one of the above defined  grid graphs, $G$ and one of its 2-factors.
Since any vertex $v \in V(G)$ is incident with exactly two edges of the 2-factor, all the possible  arrangements  of these edges around $v$ are shown in
  Figure~\ref{CvorniKod1} (the edges in bold belong to the 2-factor).
The letter assigned to any arrangement (situation) is called \emph{code letter}.

\begin{figure}[H]
\begin{center}
\includegraphics[width=3in]{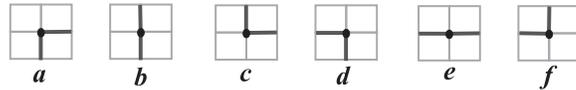}
\\ \ \vspace*{-18pt}
\end{center}
\caption{The six possible arrangements of the two edges  around any vertex with assigned   code letters.}
\label{CvorniKod1}
\end{figure}
\unskip

\begin{df} \label{def:2-factor} \cite{DjBD1,DjDB3}
For a given 2-factor of a linear or circular grid graph $G$  of width $m$  and with $m \cdot n$ vertices ($m,n \in N$),
the {\bf {\em code matrix}} $ \ds \left[ \alpha_{i,j}\right]_{m \times n}$  is
a matrix of order $m \times n$ with entries from $\{ a,b, c, d, e,f \}$ where
$\alpha_{i,j}$ is the code letter for the $i$-th vertex in $j$-th
column of $G$.
\end{df}

By reading  each column of the code matrix from top to down, we obtain  a word over alphabet $\{ a,b,c,d,e,f \}$ of length $m$, named  \emph{ alpha-word}.
When $G$ is circular grid graph, then we treat these words as circular ones and  the letter $ \alpha_{m+1,j} \stackrel{\rm def}{=}  \alpha_{1,j}$ follows the letter  $\alpha_{m,j}$.
The possibility that two vertices are adjacent in the assigned 2-factor is expressed through the two  auxiliary digraphs  ${\cal D}_{ud}$ and ${\cal D}_{lr}$
depicted  in   Figure~\ref{CvorniKod2}.

\begin{figure}[H]
\begin{center}
\includegraphics[width=3in]{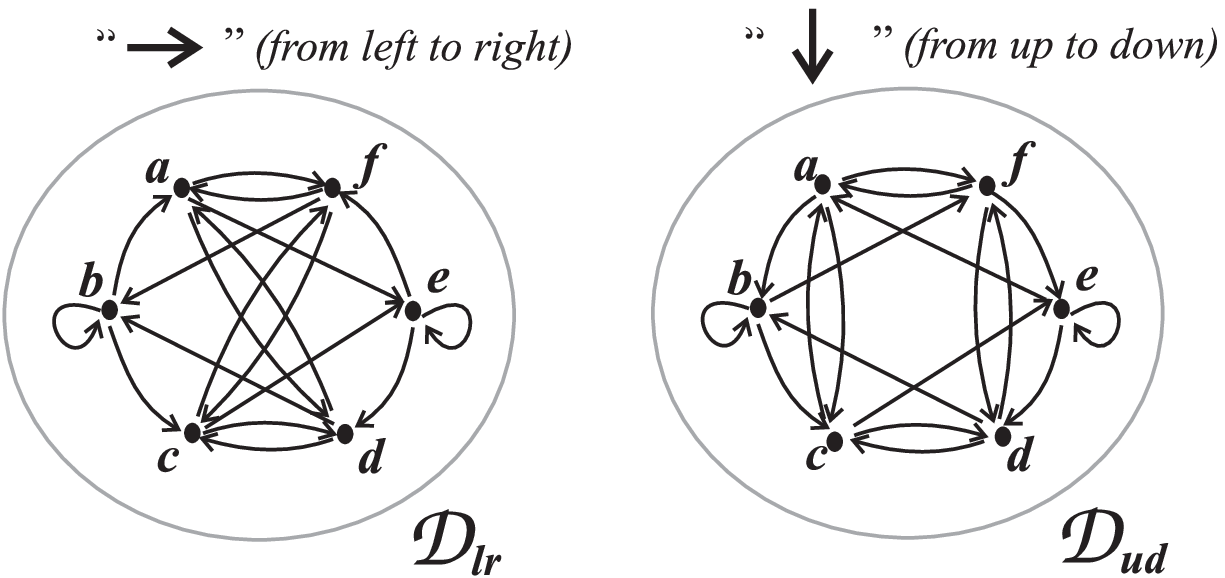}
\\ \ \vspace*{-18pt}
\end{center}
\caption{The  digraphs ${\cal D}_{ud}$ and  ${\cal D}_{lr}$.}
\label{CvorniKod2}
\end{figure}
\unskip

For each alpha-letter $\alpha$, we denote by $\overline{\alpha}$ the alpha-letter
of the situation from  Figure~\ref{CvorniKod1}  obtained
by applying reflection  symmetry with the horizontal  axis as its line of symmetry. Precisely,
  $\ds
\overline{a}  \stackrel{\rm def}{=}  c, \overline{b}  \stackrel{\rm def}{=}  b, \overline{c}  \stackrel{\rm def}{=}  a$,
 $ \overline{d}  \stackrel{\rm def}{=}  f, \overline{e}  \stackrel{\rm def}{=}  e$ and  $\overline{f}  \stackrel{\rm def}{=}  d$.
 Further, for any alpha-word $v = \alpha_{1}\alpha_{2} \ldots \alpha_{m} \in V({\cal D}_{m})$, we introduce   $\overline{v} \stackrel{\rm def}{=}\overline{\alpha}_m \overline{\alpha}_{m-1} \ldots \overline{\alpha}_1 \in V({\cal D}_{m})$.

\begin{thm}   \label{thm:stara} (The characterization of a 2-factor \cite{DjBD1, DjDB3}) \\
The code matrix  $ \ds \left[ \alpha_{i,j}\right]_{m \times n}$ for a given 2-factor of a  grid graph $G$ of width $m$ ($m \in N$)
 has the following properties:

\begin{enumerate}
\item \textbf{Column conditions:} \\ For every fixed \ $j$ ($1 \leq j \leq n$),

\begin{enumerate}
 \item \ if $G$ is linear (circular) grid graph, than the ordered pairs \ $ (\alpha_{i,j}, \alpha_{i+1,j})$
  must be arcs in the digraph \ ${\cal D}_{ud}$  for $1 \leq i \leq m-1$ ($1 \leq i
\leq m$).

\item  if $G$ is linear grid graph, than $ \alpha_{1,j} \in \{ a, d, e \}$ \ and
\ $\alpha_{m,j} \in \{ c, e, f \}$.
             \end{enumerate}

\item \textbf{Adjacency of column condition:} \\  For every  fixed  $j$ ($1 \leq j \leq n-1 $),
 the ordered pairs \ $ (\alpha_{i,j}, \alpha_{i,j+1})$ must be arcs in the digraph \ ${\cal D}_{lr}$  for $1
\leq i \leq m$.

\item \textbf{First and Last Column conditions:}
\begin{enumerate}
 \item
If $G= RG_{m}(n)$ or $G= TnC_{m}(n)$, then
the alpha-word of the first  column consists of the letters from the
set \ $\{ a, b, c \}$
 and of  the last column of the letters
  from the set \ $\{  b, d, f \}$.

\item
If $G= TkC_{m}(n)$,  then
 the ordered pairs \ $ (\alpha_{i,n}, \alpha_{i,1})$, \ where \ $1
\leq i \leq m$, \ must be arcs in the digraph \ ${\cal D}_{lr}$.

\item
If $G= MS_{m}(n)$, then
 the ordered pairs \ $ (\overline{\alpha}_{i,n}, \alpha_{m -i+1,1})$,  \ where \  $1 \leq i \leq m$,  must be arcs in the digraph \ ${\cal D}_{lr}$.

 \item
If $G= TG^{(p)}_{m}(n)$,  then
 the ordered pairs \ $ (\alpha_{i+p,n}, \alpha_{i,1})$, \ where \ $1
\leq i \leq m$, \ must be arcs in the digraph \ ${\cal D}_{lr}$.

\item
If $G= KB^{(p)}_{m}(n)$, then
 the ordered pairs \ $ (\alpha_{m+p+1-i,n}, \overline{\alpha}_{i,1})$, \ where \ $1
\leq i \leq m $, \ must be arcs in the digraph \ ${\cal D}_{lr}$.
 \end{enumerate}
\end{enumerate}
The converse, for every matrix $[\alpha_{i,j}]_{m \times n }$ with entries from $\{a,b,c,d,e,f\}$
that satisfies conditions 1--3 there is  a  unique 2-factor  on the considered grid graph $G$.
\end{thm}

This assertion enables that for considered grid graph $G$ and for fixed $m$ ($m \in N$), the counting of such code matrices (in fact all 2-factors of $G$) is  reduced to  the counting of all directed walks
in an auxiliary  digraph. If $G$ is a linear grid graph, then we label this digraph  by  $ {\cal D}_{L,m}
\stackrel{\rm def}{=} $  $(V({\cal D}_{L,m}), E({\cal D}_{L,m}))$, otherwise, if $G$ is a circular one by  $ {\cal D}_{C,m}
\stackrel{\rm def}{=} $  $(V({\cal D}_{C,m}), E({\cal D}_{C,m}))$.
By agreement, in what follows, when the class to which $G$ belongs is not specified, we label
 this digraph and corresponding  sets  without the letter $L$ or $C$ in subscript.
The set of its  vertices  \ $ V({\cal D}_{m}) $  \ consists of all  possible alpha-words, i.e. the words $\alpha_{1,j}\alpha_{2,j} \ldots \alpha_{m,j}$ over  alphabet $ \{ a,b,c,d,e,f \}$  which fulfill  {\em  Column conditions}. An arc $(v,u) \in E({\cal D}_{m})$
 joins  $ v = \alpha_{1,j}\alpha_{2,j} \ldots \alpha_{m,j}$ to  \ $ u= \alpha_{1,j+1}\alpha_{2,j+1} \ldots \alpha_{m,j+1}$, i.e.  $ v \rightarrow u $
   if and only if the {\em Adjacency of column condition} is satisfied for the ordered pair $(v,u)$
(i.e. the vertex $v$  can be the previous column for the vertex $u$ in  the  code  matrix  $ \left[ \alpha_{i,j}\right]_{m \times n}$ for a 2-factor of $G$).

\begin{exm}  \label{exm:0}
For the  Hamiltonian cycle depicted in  Figure~\ref{primer} (a),
the  columns of the code matrix (reading from left to right)  are the words $bfdb$, $cabb$ and  $dfac$.
 Similarly,  the  columns in  Figure~\ref{primer} (b) are  $bfdb$, $cabb$ and $feab$.
 Both graphs $KB^{(1)}_4(3)$ and  $TG^{(0)}_4(3)$ in this figure  has the same transfer digraph  $ {\cal D}_{C,4}$.
  In it, the first  2-factor  corresponds to  the  directed walk  of length two: $\alpha_1\alpha_2\alpha_3\alpha_4 \stackrel{\rm def}{=}$
  $ bfdb \rightarrow $ $ cabb  \rightarrow   \beta_1 \beta_2 \beta_3 \beta_4 \stackrel{\rm def}{=}dfac$, where starting and finishing vertices  fulfill $\beta_1 \beta_2 \beta_3 \beta_4 = dfac \rightarrow \overline{\alpha}_1\overline{\alpha}_4 \overline{\alpha}_3 \overline{\alpha}_2 = bbfd $.
  The second one   corresponds to
   the closed  directed walk  of length $3$:   $bfdb \rightarrow cabb \rightarrow  feab   \rightarrow bfdb$.
\end{exm}

 Clearly, $ {\cal D}_{L,m}$ is subdigraph of $ {\cal D}_{C,m}$ treating each alpha-word from $V({\cal D}_{L,m}) $ as circular one in $V({\cal D}_{C,m}) $.
For these digraphs  we have that $ \ds  \mid V({\cal D}_{C,m}) \mid = 2 \mid V({\cal D}_{L,m}) \mid = 3^m + (-1)^m$ and that both digraphs $ {\cal D}_{C,m}$ and $ {\cal D}_{L,m}$ are disconnected where  $m \geq 2$ \cite{DjBD1,DjDB3}.

\begin{df} \label{def:outlet} \cite{DjBD1,DjDB3}
The {\bf {\em outlet (inlet) word}} of  a vertex $\alpha \equiv \alpha_1 \alpha_2 \ldots \alpha_m \in V({\cal D}_{m})$ is  the  binary word
 $o(\alpha) \equiv o_1o_2 \ldots o_m$ ($i(\alpha) \equiv i_1i_2 \ldots i_m$ ), where   \bc $ \ds o_j
\stackrel{\rm def}{=} \left \{
\begin{array}{cc}{}
0, & \; \; if \; \; \alpha_j \in \{ b, d, f \} \\
1, & \; \; if \; \; \alpha_j \in \{ a, c, e \}
\end{array}
 \right.
 $ \ and \ $
 \ds i_j
\stackrel{\rm def}{=} \left \{
\begin{array}{cc}{}
0, & \; \; if \; \; \alpha_j \in \{ a, b, c \} \\
1, & \; \; if \; \; \alpha_j \in \{ d, e,f \}
\end{array}
 \right.  ,  \; \; \; 1 \leq j \leq m.
$ \ec

\noindent
For a  binary word $v \equiv b_1b_2 \ldots b_{m-1}b_{m} \in \{ 0,1\}^m$,  \bc
$\overline{v} \stackrel{\rm def}{=}  b_mb_{m-1}$ $ \ldots b_{2}b_{1}$ \ and \
$\rho (v) \stackrel{\rm def}{=} b_2 \ldots b_{m-1} b_{m}b_1$. \ec
\end{df}
\begin{exm}  \label{exm:1}
For the 2-factor    of $KB^{(1)}_4(3)$   depicted in  Figure~\ref{primer} (a)
the outlet words for  the first three columns  are $0^4$, $1100$  and $0011$, respectively.
 Similarly, for the  Hamiltonian cycle of $TG^{(0)}_4(3)$  in  Figure~\ref{primer} (b)  they   are $0^4$, $1100$  and $0110$, respectively.
\end{exm}

The   digraph \ $ {\cal D}^*_{m} \stackrel{\rm def}{=}(V({\cal D}^*_{m}), E({\cal D}^*_{m}))$ is obtained
by gluing all  the vertices from  \ $ V({\cal D}_{m})$ \  \ having the same corresponding outlet word (this word  becomes the vertex in new digraph)
 and replacing all the arcs from  $ E({\cal D}_{m})$ starting from these glued vertices and ending
at the   same vertex with only one arc.
It is proved \cite{DjBD1} that  every binary word from $ \{ 0,1 \}^m$ except the word  $(01)^k0$ when  $m=2k+1$ ($k \in N $) belongs to  $V({\cal D}^*_{L,m})$.
  For ${\cal D}^*_{C,m}  \stackrel{\rm def}{=}(V({\cal D}^*_{C,m}), E({\cal D}^*_{C,m}))$,  the set $V({\cal D}^*_{C,m})$ consists of all binary words of length $m$ \cite{DjDB3}.
Both digraphs ${\cal D}^*_{L,m}$ and ${\cal D}^*_{C,m}$ are disconnected for $m \geq 2$ (the adjacent vertices must have  the numbers of $1$s of the same parity).
Each  component of ${\cal D}^*_{L,m}$ or ${\cal D}^*_{C,m}$ is a strongly connected digraph, i.e.
their  adjacency matrices  ${\cal T}^*_{L,m}$ and   ${\cal T}^*_{C,m}$ are  symmetric binary matrices.  While the elements of the first matrix are from the set $\{ 0,1 \}$, in the second they are from the set $\{  0,1,2 \}$.

\begin{thm}   \label{thm:stara1}(\cite{DjBD1}) \
If  $f_m^{RG}(n)$,  $f_{m}^{TkC}(n)$ and  $f_{m}^{MS}(n)$ ($ m \geq 2$) denote the number of 2-factors of $RG_{m}(n)$, $TkC_{m}(n)$ and  $MS_{m}(n)$, respectively, then \
$$ \ds
f_m^{RG}(n)=  a_{1,1}^{(n)}, $$
\bc
 $ \ds
f_{m}^{TkC}(n) = \ds \sum_{\begin{array}{c} v_i \in   V({\cal D}_{L,m}^*)\end{array} }  a_{i,i}^{(n)} $
 \ \  and \ \ $ \ds
f_{m}^{MS}(n) =
   \sum_{\begin{array}{c} v_i, v_j  \in   V({\cal D}_{L,m}^*) \\
\overline{v_i}= v_{j} \end{array} } a_{i,j}^{(n)} , $
\ec
 where      $v_1 \equiv 0^m$ (corresponding to the first row and first column of ${\cal T}^*_{L,m}$) and
 $ a_{i,j}^{(n)}$ denotes the $(i,j)$-entry of $n$-th power of ${\cal T}^*_{L,m}$.
 \end{thm} \noindent

 \begin{thm}   \label{thm:stara2}(\cite{DjDB3}) \
If  $f_m^{TnC}(n)$,  $f_{m,p}^{TG}(n)$ and  $f_{m,p}^{KB}(n)$ ($ m \geq 2$) denote the number of 2-factors of $TnC_{m}(n)$, $TG^{(p)}_{m}(n)$ and  $KB^{(p)}_{m}(n)$, respectively, then
 $$ \ds
f_m^{TnC}(n)
=  a_{1,1}^{(n)}, \;   $$
\bc
 $ \; \;  \ds
f_{m,p}^{TG}(n) =  \sum_{\begin{array}{c} v_i, v_j  \in   V({\cal D}_{C,m}^*) \\
v_i= \rho^{p}(v_{j}) \end{array} } a_{i,j}^{(n)} \;  $ \ \ and \  \  $ \ds  \; \;
f_{m,p}^{KB}(n) = \sum_{\begin{array}{c} v_i, v_j  \in   V({\cal D}_{C,m})^* \\
\overline{v_i}= \rho^{p}(v_{j}) \end{array} } a_{i,j}^{(n)}$,
 \ec
   where      $v_1 \equiv 0^m$ (corresponding to the first row and first column of ${\cal T}^*_{C,m}$) and
 $ a_{i,j}^{(n)}$ denotes the $(i,j)$-entry of $n$-th power of ${\cal T}^*_{C,m}$.
\end{thm} \noindent

 By implementation of the algorithm  for obtaining the digraphs ${\cal D}^*_m$, described above,
  the  data for $m \leq 12$  gathered in \cite{DjBD1}
   suggest the structure of ${\cal D}^*_{L,m}$ expressed in the following theorem which is proved in \cite{DjDB2}.

\begin{thm}  \label{conj:1}  \cite{DjDB2}
\\
For each  $m \geq 2 $,  the digraph  ${\cal D}^{*}_{L,m}$ has exactly $\ds \left\lfloor    \ds \frac{m}{2} \right\rfloor + 1$ components, i.e.
 $\ds {\cal D}^{*}_{L,m} = {\cal A}^*_{L,m}    \cup $ $\ds  (\bigcup_{s=1}^{\left\lfloor    \frac{m}{2} \right\rfloor }{\cal B}^{*(s)}_{L,m})$, where
 $\ds \mid V({\cal B}^{*(1)}_{L,m}) \mid  \geq
  \mid V({\cal B}^{*(2)}_{L,m}) \mid  \geq  \; \; \;  \ldots \; \; \;  \geq \mid V( {\cal B}^{*(\lfloor    m/2 \rfloor )}_{L,m}) \mid $
 and   ${\cal A}^*_{L,m} $ is the one containing $1^m$.
All the components  ${\cal B}^{*(s)}_{L,m}$ ($ 1 \leq s \leq \ds \left\lfloor    \ds \frac{m}{2} \right\rfloor $) are bipartite digraphs.
\\
 If  $m$ is  odd, then $\ds \mid V({\cal B}^{*(s)}_{L,m}) \mid  =\ds  {m + 1 \choose   (m+1)/2 -s} \mbox{ \  and \ } \ds \mid V({\cal A}^{*}_{L,m}) \mid =  \ds {m  \choose (m-1)/2}.$
\\
  If  $m$ is even, then $\ds \mid V({\cal B}^{*(s)}_{L,m}) \mid  = \ds 2 {m \choose   m/2 -s} \mbox{ \  and \  } \ds \mid V({\cal A}^{*}_{L,m}) \mid =  \ds {m \choose m/2}.$ \\
    The vertices $v$ and $\overline{v}$ belong to  the same component. When the component is bipartite they are placed in the same  class  if and only if  $m$ is odd.
\end{thm}

For the digraph ${\cal D}^*_{C,m}$, the component which contains $1^m$  is  marked  by
 ${\cal A}^*_{C,m} $, while  the one  containing $0^m$ by ${\cal N}^*_{m} $ (the one responsible for counting 2-factors for thin cylinder grid graphs).
 The former is produced from the component ${\cal A}_{C,m} $ of ${\cal D}_{C,m}$ which contains the vertex $e^m$, and the latter one from the component ${\cal N}_{m} $
 which contains the vertex $b^m$.
Data for $m \leq 10$  gathered in \cite{DjDB3}
   suggested the structure of ${\cal D}^*_{C,m}$ expressed in the following theorem.

\begin{thm} \label{thm:hip} (MAIN THEOREM)  \cite{DjDB3} \\
For each even  $m\geq 2 $,
 the digraph  ${\cal D}^{*}_{C,m}$ has exactly $\ds \left\lfloor    \ds \frac{m}{2} \right\rfloor + 1$ components, i.e.
 $\ds {\cal D}^{*}_{C,m} =$ $ {\cal A}^*_{C,m}    \cup $ $\ds  (\bigcup_{s=1}^{\left\lfloor    \frac{m}{2} \right\rfloor }{\cal B}^{*(s)}_{C,m})$, where
  ${\cal A}^*_{C,m} $  contains both  $1^m$ and $0^m$, all the components  ${\cal B}^{*(s)}_{C,m}$ ($ 1 \leq s \leq \ds \left\lfloor    \ds \frac{m}{2} \right\rfloor $) are bipartite digraphs,
  $\ds \mid V({\cal B}^{*(s)}_{C,m}) \mid  = \ds 2 {m \choose   m/2 -s} \mbox{ \  and \  } \ds \mid V({\cal A}^{*}_{C,m}) \mid =  \ds \ds {m \choose m/2}.$ \\ \\
For each odd  $m\geq 1 $,  the digraph  ${\cal D}^{*}_{C,m}$ has exactly two  components, i.e.
 ${\cal D}^{*}_{C,m} = {\cal A}^*_{C,m}    \cup {\cal N}^*_{m} $,  which are mutually  isomorphic and
  with $2^{m-1}$ vertices.
 \end{thm}

The  aim of  this paper is the first proof of Theorem~\ref{thm:hip}.
In the next section, we prove the main theorem.

\begin{center}
{\bf 2. PROOF OF  THE MAIN THEOREM}
\end{center}

\begin{df} \label{def:ch}  \cite{DjDB2}
The total number of 0's at odd (even) positions in a binary word $x$ of length $m$ ($ m \in N$)
is denoted by $odd(x)$ ($even(x)$).
\bc
$Z(x) \stackrel{\rm def}{=} odd(x) - even(x)$. \ec
\noindent
The set  $S_{m}^{(0)}$ ($m \in N$) consists of all the  binary $m$-words whose number of 0's at odd positions is equal to the  number of 0's at even  positions. \\
For  $1 \leq s \leq \lfloor m/2 \rfloor$,  $S_m^{(s)}\stackrel{\rm def}{=}  R_m^{(s)} \cup  G_m^{(s)}$  where
the sets $R_{m}^{(s)}$  and $G_{m}^{(s)}$ consist of  all the  binary words $x$ of the length $m$
for which  $ Z(x)  = s$ and   $ Z(x) = -s$, respectively. Additionally, if  $m$ is odd, then $R_m^{(\lceil m/2 \rceil)}\stackrel{\rm def}{=}\{ 0 (10)^{\lfloor m/2 \rfloor} \}$.
\end{df}

Note that $\ds \bigcup_{s=0}^{\lfloor m/2 \rfloor} S_m^{(s)}  =  V({\cal D}_{L,m}^*)$, where $ V({\cal D}_{L,m}^*) = \{ 0,1\}^m$ for $m$-even, and
 $ V({\cal D}_{L,m}^*) = \{ 0,1\}^m \backslash R_m^{(\lceil m/2 \rceil)} = \{ 0,1\}^m \backslash \{ 0 (10)^{\lfloor m/2 \rfloor} \}$ for $m$-odd.
In \cite{DjDB2}, it is proved that the subdigraphs of ${\cal D}^{*}_{L,m}$ induced  by the sets  $S_{m}^{(s)}$ ($0 \leq s \leq \lfloor m/2 \rfloor$)
are its  components, i.e. $ \langle S_{m}^{(0)} \rangle_{{\cal D}^{*}_{L,m}} = {\cal A}^*_{L,m}$  and  $ {\cal B}^{*(s)}_{L,m}  = \langle S_{m}^{(s)} \rangle_{{\cal D}^{*}_{L,m}} $,  where $s=1,2, \ldots ,  \lfloor m/2 \rfloor$. Each component $ {\cal B}^{*(s)}_{L,m} $ is a bipartite digraph and the  sets
$ R_m^{(s)} $ and $G_m^{(s)}$ are its classes. We call  the vertices  from $ R_m^{(s)} $ and $G_m^{(s)}$ \textbf{ red } and \textbf{ green } vertices, respectively.
Some of the representatives for these sets are introduced in the following way.
\begin{df} \label{def:nova1} \cite{DjDB2}
For even $m$,  the zero-word  $Q_{m}^{(0)} \stackrel{\rm def}{=} 0^{m} \in S_{m}^{(0)} $ and  the words   $Q_{m}^{(s)} \stackrel{\rm def}{=} (01)^s0^{m-2s} \in R_{m}^{(s)}$ ($1 \leq s \leq m/2 $)
are called  the  \textbf{ queens}.
\\
For odd  $m$, the words
 $Q_{m}^{(s)} \stackrel{\rm def}{=} (01)^{s-1}0^{m-2s+2} \in R_{m}^{(s)}$ ($1 \leq s \leq \lfloor m/2 \rfloor+1$) are
called  the  \textbf{ queens}, while   the words
 $L_{m}^{(s)} \stackrel{\rm def}{=}(10)^{s+1}0^{m-2s-2} \in S_{m}^{(s)}$ ($0 \leq s < \lfloor m/2 \rfloor $) and the word
 $\ds L_{m}^{(\lfloor m/2 \rfloor)} \stackrel{\rm def}{=}(10)^{\lfloor m/2 \rfloor}1 \in S_{m}^{(\lfloor m/2 \rfloor)}$   are  called the   \textbf{ court ladies}.
\end{df}

\begin{lem}  \label{lem:1}
For each even  $m\geq 2 $,  the digraph  ${\cal D}^{*}_{C,m}$ has exactly $\ds   \ds \frac{m}{2}+ 1$ components, i.e.
 $\ds {\cal D}^{*}_{C,m} = {\cal A}^*_{C,m}   \cup ( \bigcup_{s=1}^{m/2}{\cal B}^{*(s)}_{C,m}),$
  where  the component  ${\cal A}^*_{C,m} $ is the one containing $1^m$,
 $$\ds \mid V({\cal A}^{*}_{C,m}) \mid =  \ds \ds {m \choose m/2} \mbox{ \  and \  } \ds \mid V({\cal B}^{*(s)}_{C,m}) \mid  = \ds 2 {m \choose   m/2 -s}, \; \;
 1 \leq s \leq m/2.$$
 \end{lem}
\begin{proof}
Note that ${\cal D}^{*}_{L,m}$ is subdigraph of ${\cal D}^{*}_{C,m}$. The only difference is  in added (new) arcs.
Recall that the statement of this lemma is valid if ${\cal D}^{*}_{C,m}$ is replaced with  ${\cal D}_{L,m}^*$ (Theorem~\ref{conj:1}).
 Therefore, it is sufficient to prove that two different  queens $Q_m^{(s_1)}=(01)^{s_1}0^{m-2s_1}$ and $Q_m^{(s_2)}=(01)^{s_2}0^{m-2s_2}$ with the same parity of the number $1$'s ($s_1 \equiv s_2 (mod 2)$, $s_1 \neq s_2$) are not connected in ${\cal D}^{*}_{C,m}$.

For that purpose, we suppose the opposite, i.e.  that in ${\cal D}_{C,m}^*$ there exists a directed walk $v_0 \rightarrow v_1 \rightarrow v_2 \rightarrow \ldots \rightarrow v_{k-1} \rightarrow v_k$ of length $k \in N$, where   $ v_0 = Q_m^{(s_1)}$ and $ v_k = Q_m^{(s_2)}$.
For this directed  walk, the corresponding part of the grid  induced by $m \cdot k$ vertices (the thin cylinder grid graph)  is bipartite because $m$ is even.
 The directed walk $v_0 \rightarrow v_1 \rightarrow v_2 \rightarrow \ldots \rightarrow v_{k-1} \rightarrow v_k$ determines a spanning union of paths  (open paths and   cycles) in this grid.  The ends of these open paths belong to the first or/and the last column of the  cylinder grid. Each   cycle (if exists) in  this union  has the same number of  vertices of both colors (say gray and black). However, for the union of the open paths it is not valid.
Namely, if $k$ is even  (see Figure~\ref{nepovezanost} (a)), the difference of the numbers of open paths with both ends in gray vertices and the ones in black vertices is exactly
$\mid  \ds \frac{s_1 - s_2}{2}\mid > 0$. If $k$ is odd (see Figure~\ref{nepovezanost} (b)), all $\mid  \ds \frac{s_1 + s_2}{2}\mid$ open paths have end vertices in the same color. In both cases we come in contradiction with the fact that in the considered part of the grid the numbers of vertices of both colors are equal.
\\
In this way, we obtain that ${\cal A}^*_{C,m} = \langle S_{m}^{(0)} \rangle_{{\cal D}^{*}_{C,m}}$  and  $ {\cal B}^{*(s)}_{C,m}  = \langle S_{m}^{(s)} \rangle_{{\cal D}^{*}_{C,m}} $,  where $\ds s=1,2, \ldots ,  \frac{m}{2}$.
Consequently,
$\ds \mid V({\cal A}^{*}_{C,m}) \mid =  \mid  S_{m}^{(0)}  \mid =  \mid V({\cal A}^{*}_{L,m}) \mid = \ds {m \choose m/2}$ \  and
$ \ds \mid V({\cal B}^{*(s)}_{C,m}) \mid  = $  $ \mid  S_{m}^{(s)}  \mid = $  \\
$ \ds  \mid V({\cal B}^{*(s)}_{L,m}) \mid  =  $  $\ds 2 {m \choose   m/2 -s}, \; \;  1 \leq s \leq m/2.$ $\Box$

\begin{figure}[H]
\begin{center}
\includegraphics[width=5in]{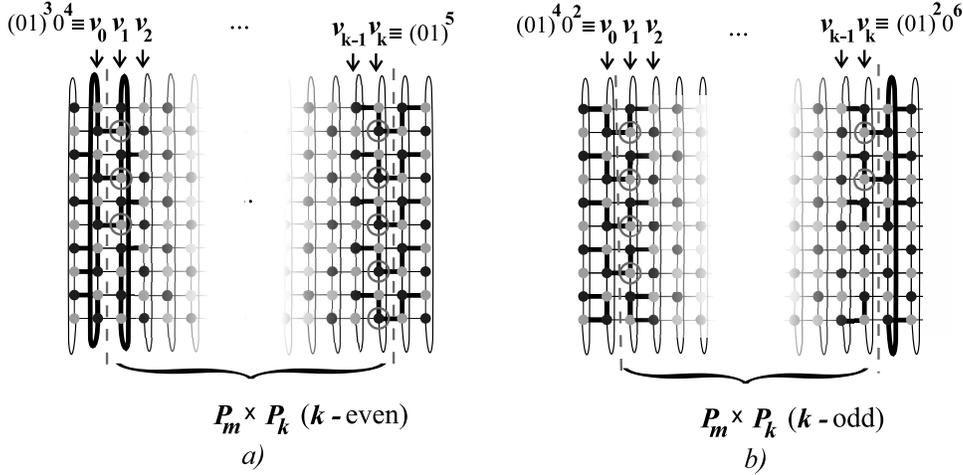}
\\ \ \vspace*{-18pt}
\end{center}
\caption{Two different queens $Q_m^{(s_1)}=(01)^{s_1}0^{m-2s_1}$ and $Q_m^{(s_2)}=(01)^{s_2}0^{m-2s_2}$ where $m$ is even,  $s_1 \equiv s_2 (mod \;  2)$ and  $s_1 \neq s_2$ are not connected in ${\cal D}^{*}_{C,m}$.}
\label{nepovezanost}
\end{figure}
\unskip
\end{proof}

\begin{lem}  \label{lem:2}
When $m$ is even, all the components  ${\cal B}^{*(s)}_{C,m}$ ($ 1 \leq s \leq \ds   \ds \frac{m}{2}  $) are bipartite digraphs.
    The vertices $v$ and $\overline{v}$ belong to  the same component but in different classes (colors).
\end{lem}
\begin{proof}
 In order to prove that  the component ${\cal B}^{*(s)}_{C,m}$ ($ 1 \leq s \leq \ds  \ds \frac{m}{2} $) is the  bipartite digraph $(R_{m}^{(s)},G_{m}^{(s)})$,  it is sufficient to prove that every  arc from the set $E({\cal B}^{*(s)}_{C,m}) \backslash E({\cal B}^{*(s)}_{L,m})$  has the  vertices in different colors (red and green).
 Consider such an arc $vw \in E({\cal B}^{*(s)}_{C,m})$. Then, there exists $\alpha = \alpha_1  \alpha_2 \ldots  \alpha_m  \in V({\cal D}_{C,m})$ for which $i(\alpha) = v= v_1v_2 \ldots v_m$,   $o(\alpha) = w=w_1w_2 \ldots w_m$,  $\alpha_1 \in \{ b, c, f \}$ and $\alpha_m \in \{ a, b, d \}$.
 Since $\alpha \neq b^m$ ($o(b^m) = 0^m  \in V({\cal A}^{*}_{C,m})$), then there exists \
  $\ds J = \min\{j \in N \; \mid \;  1 \leq j \leq m-1 \wedge (v_j = 1  \vee w_j = 1) \} .$ \
  This implies that  there exists  an  arc $\rho^J(v) \rightarrow \rho^J(w)$ in  ${\cal D}^*_{L,m}$
  where  $\rho^j(x_1x_2 \ldots x_m) \stackrel{\rm def}{=} x_{j+1}x_{j+2} \ldots x_mx_1x_2 \ldots x_j$ for any  word $x_1x_2 \ldots x_m$ of length $m$ and $1 \leq j \leq m-1$.
  Namely,   $\rho^J(\alpha) \in V({\cal D}_{L,m})$, $ i(\rho^J(\alpha)) = \rho^J(v)$ and   $o(\rho^J(\alpha)) = \rho^J(w).$
 Note that for an arbitrary binary word $x=x_1x_2 \ldots x_m$ and $1 \leq j \leq m -1$, we have that $Z(\rho^j(x)) = (-1)^j  Z(x)$.
Consequently,  both $v$ and $w$  belong to the same set $S_m^{(s)}$ ($0 \leq s \leq m/2$) which  contains $\rho^J(v)$ and $ \rho^J(w)$. Since
 the vertices $\rho^J(v)$ and $ \rho^J(w)$ belong to the different sets  $R_m^{(s)}$ and  $G_m^{(s)}$ (they belong to different classes of  the component ${\cal B}^{*(s)}_{L,m}$ of  ${\cal D}_{L,m}^*$),  it is valid for $v$ and $w$, too.
\\
The second statement of this lemma is a simple consequence from the linear case (Theorem~\ref{conj:1}). $\Box$
\end{proof}

\begin{lem}  \label{lem:3}
For each odd  $m\geq 1 $,  the digraph  ${\cal D}^{*}_{C,m}$ has exactly two  components  ${\cal A}^*_{C,m}$ and ${\cal N}^*_m $ (${\cal D}^{*}_{C,m} = {\cal A}^*_{C,m}   \cup {\cal N}^*_{m} $)  which are isomorphic and  each of them  has $2^{m-1}$ vertices.
 \end{lem}
\begin{proof}
 Recall that when $m$ is odd, the court ladies
 $L_m^{(s)}=(10)^{s+1}0^{m-2s-2} \in G_{m}^{(s)} \subseteq  S_{m}^{(s)}$ ($1 \leq s \leq \lfloor m/2  \rfloor -1$) and  $L_m^{(0)}=(10)0^{m-2} \in S_{m}^{(0)}$ fulfill  $ Z(L_m^{(s)}) = -s$ ($0 \leq s \leq \lfloor m/2  \rfloor -1$).  For the queens $Q_m^{(s)}= (01)^{(s-1)}0^{m-2s+2} \in R_{m}^{(s)} \subseteq  S_{m}^{(s)}$ ($1 \leq s \leq \lfloor m/2  \rfloor $) we have $ Z(Q_m^{(s)}) = s$. Additionally, the queen $Q_m^{( \lfloor m/2 \rfloor+1)} = (01)^{\lfloor m/2  \rfloor} 0 \notin V({\cal D}^{*}_{L,m})$  now belongs to  $V({\cal D}^{*}_{C,m})$ and $ \ds Z(Q_m^{( \lfloor m/2 \rfloor+1)}) = \lfloor m/2 \rfloor+1$.
The digraphs induced by the sets $S_m^{(s)}$, $s=0,1,2, \ldots, \lfloor m/2  \rfloor$ in ${\cal D}^{*}_{L,m}$ are  connected digraphs. Clearly,
they are subdigraphs of ${\cal D}^{*}_{C,m}$. We prove that all the vertices $v \in V({\cal D}^{*}_{C,m})$ with even $Z(v)$ belong to
${\cal A}^{*}_{C,m}$ - the component  of ${\cal D}^{*}_{C,m}$ containing $1^m$ ($Z(1^m) =0$)
while the ones  with odd $Z(v)$ to ${\cal N}^{*}_{m}$ - the component  of ${\cal D}^{*}_{C,m}$ containing $0^m$ ($Z(0^m) =1$ and the vertices $1^m$ and $0^m$ are not connected because their numbers of 1's have opposite parity).

For this purpose, note that  the number of $1$'s in the  words $L_m^{(s)}$ and $Q_m^{(s+2)}$ ($ 0 \leq s \leq \lfloor m/2  \rfloor -1$) is equal ($s+1$). They are directly connected by an arc in ${\cal D}^{*}_{C,m}$ because there exists the alpha word $\alpha = f (af)^sab^{m-2s-2} \in V({\cal D}_{C,m})$ for which $i(\alpha) = L_m^{(s)}$ and
$o(\alpha) = Q_m^{(s+2)}$.
In this way we obtain that $V({\cal A}^{*}_{m})$ and $V({\cal N}^{*}_{m})$ consist of all the circular binary words $v$ of length $m$ for which $Z(v)$ is even and odd, respectively.

The isomorphism between  ${\cal A}^{*}_{C,m}$ and ${\cal N}^{*}_{m}$ is the simple consequence of the isomorphism  between  ${\cal A}_{C,m}$ - the component of  ${\cal D}_{C,m}$
containing $e^m$ and ${\cal N}_{m}$  - the component of  ${\cal D}_{C,m}$ containing $b^m$.
Let us prove  the latter.
For this sake, we define the function $f: \{ a, b, c, d, e,f \}  \longrightarrow \{ a, b, c, d, e,f \} $
with  $\ds f: \left(  \begin{array}{c} a \; \;   b \; \;  c \; \;  d \; \;  e \; \;  f \\   f \; \;   e \; \; d \; \;  c \; \;  b \; \;  a   \end{array}  \right) .$
With direct verification
 we  conclude that  $f$ is automorphism of the digraph  ${\cal D}_{ud}$. Consequently,
 for any word  $\alpha = \alpha_1 \alpha_2  \ldots \alpha_m  \in V({\cal D}_{C,m})$, the word $f(\alpha_1)f(\alpha_2)  \ldots f(\alpha_m)$
 belongs to $V({\cal D}_{C,m})$, too.

 Now, we define the function $ \ds  F : V({\cal D}_{C,m}) \longrightarrow  V({\cal D}_{C,m})$ by $ F(\alpha) = \beta$  if and only if $\beta_i = f(\alpha_i)$, for all $i=1,2, \ldots, m$ where $\alpha = \alpha_1 \alpha_2  \ldots \alpha_m, \beta = \beta_1, \beta_2, \ldots , \beta_m \in V({\cal D}_{C,m})$.
\\ Obviously, $F$ is a  bijection and involution because the same is valid for the function $f$.
The function $f$ is also  automorphism of the digraph  ${\cal D}_{lr}$, which implies that
\bc $(\forall \alpha,  \beta   \in V({\cal D}_{C,m}))( \alpha \rightarrow \beta  \Leftrightarrow F(\alpha) \rightarrow F(\beta) ). $\ec
 Since $F(e^m)= b^m$, we conclude that ${\cal A}_{C,m}$ and ${\cal N}_{m}$ are isomorphic. Consequently, ${\cal A}^*_{C,m}$ and ${\cal N}^*_{m}$ are isomorphic.
$\Box$
\end{proof}
 Additionally, note that if  $ F(\alpha) = \beta$, where  $  o(\alpha)= v = v_1v_2 \ldots v_m$ and   $ o(\beta)= w = w_1w_2 \ldots w_m$,  then $v_i = 0$ if and only if $w_i=1$, for all $i=1,2, \ldots , m$. In this way, the isomorphism between ${\cal A}^*_{C,m}$ and ${\cal N}^*_{m}$ is determined with
 the function $c: V({\cal D}^{*}_{C,m})  \longrightarrow V({\cal D}^{*}_{C,m})$ defined by
 $c(v) =  w = w_1w_2 \ldots w_m$, where $v = v_1v_2 \ldots v_m$ and $\ds w_i = \left\{  \begin{array}{cc} 0, &  \mbox{ if }   v_i = 1 \\ 1, &  \mbox{ if }   v_i = 0
  \end{array}  \right.$. Clearly, $c(c(v)) = v$. We can say that  the words $v$ and $w$ are mutually complementary.

 \vspace*{5mm}

Lemma~\ref{lem:1}, Lemma~\ref{lem:2} and Lemma~\ref{lem:3} complete the proof of Theorem~\ref{thm:hip}.


\vspace{1.5cc}
\begin{center}
{\bf ACKNOWLEDGEMENTS}
\end{center}


This work  was  supported by  the Ministry of
Science, Technological Development and Innovation of the   Republic of Serbia \ (Grants $ 451-03-9/2022-14/200125$, $451-03-68/2022-14/200156$)
and the Project of the Department for fundamental disciplines in technology, Faculty of Technical Sciences, University of Novi Sad "Application of general disciplines in technical and IT sciences".



\noindent Faculty of Technical Sciences,
  University of Novi Sad,
  Novi Sad, Serbia\\
     E-mail: jelenadjokic@uns.ac.rs  \\
E-mail: ksenija@uns.ac.rs 

\vspace*{0.5cm}

 \noindent
  Dept.\ of Math.\ \&\ Info.,
  Faculty of Science,
  University of Novi Sad,
  Novi Sad, Serbia \\
   E-mail: olga.bodroza-pantic@dmi.uns.ac.rs

\end{document}